\documentclass{amsart}
\usepackage{mathrsfs,latexsym,amsfonts,amssymb}

\newtheorem{theorem}{Theorem}[section]
\newtheorem{lemma}[theorem]{Lemma}
\newtheorem{corollary}[theorem]{Corollary}
\newtheorem{question}[theorem]{Question}
\theoremstyle{definition}
\newtheorem{definition}[theorem]{Definition}
\newtheorem{problem}[theorem]{Problem}
\newtheorem{proposition}[theorem]{Proposition}
\theoremstyle{remark}

\newtheorem{example}[theorem]{Example}

\begin{document}
\title[Suitable sets for topological groups revisited]
{Suitable sets for topological groups revisited}

\author{Fucai Lin*}
\address{(Fucai Lin): School of mathematics and statistics,
Zhangzhou Normal University, Zhangzhou 363000, P. R. China}
\email{linfucai2008@aliyun.com; linfucai@mnnu.edu.cn}

\author{Jiamin He}
\address{(Jiamin He): School of mathematics and statistics, Minnan Normal University, Zhangzhou 363000, P. R. China}
\email{hjm1492539828@163.com}

\author{Jiajia Yang}
\address{(Jiajia Yang): School of mathematics and statistics, Minnan Normal University, Zhangzhou 363000, P. R. China}
\email{3561333253@qq.com}

\author{Chuan Liu}
\address{(Chuan Liu): Department of Mathematics,
Ohio University Zanesville Campus, Zanesville, OH 43701, USA}
\email{liuc1@ohio.edu}

\thanks{The authors are supported by Fujian Provincial Natural Science Foundation of China (No: 2024J02022) and the NSFC (Nos. 11571158).}

\keywords{suitable set; topological group; free topological group; linearly orderable topological group; $\omega^{\omega}$-base; separable space; $k$-space.}
\subjclass[2000]{22A15, 54F05, 54H11, 54H99}

\begin{abstract}
A  discrete subset $S$ of a topological group $G$ is called a {\it suitable set} for $G$ if $S\cup \{e\}$ is closed in $G$ and the subgroup generated by $S$ is dense in $G$, where $e$ is the neutral element of $G$. In this paper, the existence of suitable sets in topological groups is studied. It is proved that, for a non-separable $k_{\omega}$-space $X$ without
non-trivial convergent sequences, the $snf$-countability of $A(X)$ implies that $A(X)$ does not have a suitable set, which gives a partial answer to \cite[Problem 2.1]{TKA1997}. Moreover, the existence of suitable sets in some particular classes of linearly orderable topological groups is considered, where Theorem~\ref{t4}
provides an affirmative answer to \cite[Problem 4.3]{ST2002}. Then, topological groups with an $\omega^{\omega}$-base are discussed, and every linearly orderable topological group with an $\omega^{\omega}$-base being metrizable is proved; thus it
has a suitable set. Further, it follows that each topological group $G$ with an $\omega^{\omega}$-base has a suitable set whenever $G$ is a $k$-space,
which gives a generalization of a well-known result in \cite{CM}. Finally, some cardinal invariant of topological groups with a suitable set are provided. Some results of this paper give some partial answers to some open problems posed in~\cite{DTA} and~\cite{TKA1997} respectively.
\end{abstract}

\maketitle

\section{Introduction}
Suitable sets were considered in the contexts of Galois cohomology and free profinite groups by Tate (see~\cite{Doa}) and Mel'nikov~\cite{Mel} respectively. Later Hofmann
and Morris promoted the concept of suitable set in their
well-known monograph on compact groups and in their seminal paper~\cite{HM1990}.
Thereafter many topologists studied suitable sets in topological groups,
see, for instance, ~\cite{CG-F}, \cite{CM}, \cite{DTT} and \cite{TKA1997}.
Fundamental results were obtained by Comfort et al. in~\cite{CM}
and Dikranjan et al. in~\cite{DTT} and~\cite{DTT2000}.
Many examples of topological groups with or without suitable sets were constructed in~\cite{TKA1997}.

A subset $S$ of a topological group $G$ is a {\it suitable set} for $G$, if
$S$ is a discrete subset of $G$, $S\cup \{e\}$ is closed, and
the subgroup $\langle S\rangle$ of $G$ generated by $S$ is dense in $G$, see~\cite{HM1990}.
Let $\mathscr{S}$ (respectively, $\mathscr{S}_{c}$) be the class of topological groups $G$
having a suitable (respectively, closed suitable) set.
It turns out that very often a suitable set of a group $G$ generates $G$.
This fact suggests to devote a special attention to classes
$\mathscr{S}_g$ (respectively, $\mathscr{S}_{cg}$) of topological groups $G$
having a suitable (respectively, closed suitable) set which generates $G$.

In this paper, we mainly consider the following three open problems, which were posed more than twenty years ago. We shall give some partial answers to Problems~\ref{pr0} and~\ref{pr1}, and an affirmative answer to Problem~\ref{pr6}.

\begin{problem}\cite[Problem 2.1]{TKA1997}\label{pr0}
Let $X$ be a non-separable compact space without
non-trivial convergent sequences. Is it true that $A(X)\not\in\mathscr{S}$?
\end{problem}

\begin{problem}\cite[Problem 1.5]{TKA1997}\label{pr1}
Does every linearly orderable topological group
have a (closed) suitable set?
\end{problem}

Let the direct product $\prod_{\alpha<\tau}G_{\alpha}$ endowed with the Tychonoff topology, and let $G$ the $\sigma$-product of $\prod_{\alpha<\tau}G_{\alpha}$ which inherits the subspace group topology $\mathscr{F}_{G}$, where each $G_{\alpha}$ is a non-trivial discrete group. For each $g=(g_{\alpha})_{\alpha<\tau}\in \prod_{\alpha<\tau}G_{\alpha}$, we denoted the set $\{g_{\alpha}: g_{\alpha}\neq e_{\alpha}, \alpha<\tau\}$ by $Proj(g)$.

In \cite{ST2002}, M. Sanchis and M. Tkachenko posed the following open problem.

\begin{problem}\cite[Problem 4.3]{ST2002}\label{pr6}
It is true that every subgroup of $(G, \mathscr{F}_{G})$ has a (generating) suitable
set?
\end{problem}

This paper is organized as follows. In Section 2, we introduce the
necessary notations and terminology which are used in
the paper. In Section 3, we mainly give a partial answer to Problem~\ref{pr0}, and consider some particular class of free topological groups with a suitable set. In Section 4, we mainly provide some partial answer to Problem~\ref{pr1}, and give an affirmative answer to Problem~\ref{pr6}. In Section 5, we prove that, for any topological group $G$ with an $\omega^{\omega}$-base, if $G$ is a $k$-space then it has a suitable base. Moreover, we show that every linearly orderable topological group with an $\omega^{\omega}$-base is metrizable; thus it
has a suitable set. In Section 6, we mainly consider the the density and cardinality of a closed suitable set in a topological group.
\maketitle

\section{Preliminaries}
All spaces are Tychonoff unless stated otherwise. We denote by $\omega$ and $\mathbb{N}$ the first infinite ordinal and the set of all natural
numbers respectively. Given a group $G$, let $e_G$ denote the neutral element of $G$; if no
  confusion occurs, we simply use $e$ instead of $e_G$ to denote the neutral
  element of $G$. Readers may refer
\cite{A2008, G1984} for notations and terminology not
explicitly given here.

Let $F(X)$ be the free topological
group and $A(X)$ be the free Abelian topological group over a
space $X$ respectively. Moreover, We always use $G(X)$ to denote
  topological groups $F(X)$ or $A(X)$. Each element $g\in F(X)$ distinct from the neutral element can be uniquely written
in the form $g=x_{1}^{r_{1}}x_{2}^{r_{2}}\cdots x_{n}^{r_{n}}$, where
$n\geq 1$, $r_{i}\in\mathbb{Z}\setminus\{0\}$, $x_{i}\in X$, and
$x_{i}\neq x_{i+1}$ for each $i=1, \cdots, n-1$, and the {\it support}
of $g=x_{1}^{r_{1}} x_{2}^{r_{2}}\cdots x_{n}^{r_{n}}$ is defined as
$\mbox{supp}(g) :=\{x_{1}, \cdots, x_{n}\}$. Given a subset $K$ of
$F(X)$, we put $\mbox{supp}(K):=\bigcup_{g\in K}\mbox{supp}(g)$. Similar
assertions (with the obvious changes for commutativity) are valid for
$A(X)$.

\begin{definition}
Let $\mathscr P$ be a cover of a space $X$ such that (i) $\mathscr P =
  \bigcup_{x\in X}\mathscr{P}_{x}$; (ii) for each point $x\in X$, if $U,V\in
  \mathscr{P}_{x}$, then $W\subseteq U\cap V$ for some $W\in \mathscr{P}_{x}$;
  and (iii) for each point $x\in X$ and each open neighborhood $U$ of $x$ there
  is some $P\in\mathscr P_x$ such that $x\in P \subseteq U$.

$\bullet$  Then, $\mathscr P$
  is called an \emph{sn-network} \cite{Lin1996} for $X$ if, for each point $x\in X$,
  each element of $\mathscr P_x$ is a sequential neighborhood of $x$ in $X$.

$\bullet$  $X$ is called \emph{snf-countable} \cite{Lin1996} if $X$ has an $sn$-network
  $\mathscr P$ such that $\mathscr P_x$ is countable for any $x\in X$.
  \end{definition}

\begin{definition}
Let $X$ be a space.

$\bullet$ The space $X$ is a \emph{$k$-space} provided that a subset $C\subseteq X$ is closed in $X$ if
  $C\cap K$ is closed in $K$ for each compact subset $K$ of $X$.

$\bullet$  The space $X$ is called a {\it $k_{\omega}$-space} if there exists a family of countably many compact subsets $\{K_{n}: n\in\mathbb{N}\}$ of $X$ such that each subset $F$ of $X$ is closed in $X$ provided that $F\cap K_{n}$ is closed in $K_{n}$ for each $n\in\mathbb{N}$. Clearly, each $k_{\omega}$-space is a $k$-space.

$\bullet$ The density $d(X)$ of $X$ is defined as the smallest cardinal number of the form $|A|$ for each dense subset $A$ of $X$. If $d(X)\leq\omega$, then we say that $X$ is {\it separable}.

$\bullet$ The space $X$ is said to be {\it Lindel\"{o}f} if each open cover of $X$ has a countable subcover.

$\bullet$ The {\it pseudocharacter $\psi(X)$} of $X$ is the smallest
infinite cardinal $\kappa$ such that any point of $X$ is an intersection of at most $\kappa$
open subsets of $X$.

$\bullet$ The {\it extent $e(X)$}
is the supremum of cardinalities of closed discrete subspaces of $X$.

\end{definition}

\begin{definition} Let $X$ be space. Then

$\bullet$ A family $\mathcal{P}$ of subsets of
$X$ is called a {\it network} for $X$ if for each $x\in X$ and neighborhood $U$ of $x$ there exists $P\in \mathcal{P}$ such that $x\in P\subset U$.

$\bullet$ Recall that $X$ is a {\it $\sigma$-space} if it is regular and has a $\sigma$-locally finite network.
\end{definition}

\begin{definition}
	A topological space $(X, \mathcal{T})$ is called \emph{linearly orderable} if there exists a total order ``$\le$'' on $X$ for which the set of all open rays
	\[
	\{(a, +\infty) \mid a \in X\} \cup \{(-\infty, b) \mid b \in X\}
	\]
   forms a subbase for the topology $\mathcal{T}$, where $(a,+\infty)=\{x\in X:x>a\}$ and $(-\infty,b)=\{x\in X:x<b\}$.
\end{definition}

A topological group $(G,\mathcal{T})$ is called to be  {\it topologically orderable group} if there is a total order ``$\le$'' on $G$ such that the order topology induced by $\le$ coincides with the topology $\mathcal{T}$.

\maketitle

\section{Free topological groups with a suitable set}
In this section, we mainly give a partial answer to Problem~\ref{pr0}, and consider some particular class of free topological groups with a suitable set. First we give a technical lemma.

By the proof (g) of \cite[Theorem 7.10.13]{A2008}, we have the following lemma.

\begin{lemma}\label{p0}
Let $X$ be a space. If $S$ is dense in $A(X)$, then supp($S$) is dense in $X$.
\end{lemma}

The following theorem gives a sufficient condition such that $A(X)$ does not have a suitable set for any non-separable $k_{\omega}$-space $X$ without
non-trivial convergent sequences.

\begin{theorem}\label{tttt}
Let $X$ be a non-separable $k_{\omega}$-space without
non-trivial convergent sequences. If $A(X)$ is $snf$-countable, then $A(X)$ does not have a suitable set.
\end{theorem}

\begin{proof}
Let $A(X)$ have a suitable set $S$. Then $S$ is uncountable. Indeed, it is obvious that $\langle S\rangle$ is dense in $A(X)$, hence supp($S$) is dense in $X$ by Lemma~\ref{p0}. Then $S$ is a uncountable set since $X$ is a non-separable. Therefore, $S$ is uncountable. Since $X$ is a $k_{\omega}$-space, it follows that $A(X)$ is a $k_{\omega}$-space by \cite[Theorem 7.4.1]{A2008}. Then $S\cup\{e\}$ is a $k_{\omega}$-subspace since it is closed in $A(X)$. Then it is easy to check that $S\cup\{e\}$ contains a subspace $S_{0}\cup\{e\}$ which is homeomorphic to one-point compactification of a uncountable discrete space. Hence $A(X)$ contains a non-trivial convergent sequence.
However, since $A(X)$ is $snf$-countable, it follows from \cite[Theorem 3.1]{FLC} that $A(X)$ contains no non-trivial convergent sequences, which is a contradiction. Therefore, $A(X)$ does not have a suitable set.
\end{proof}

It follows from \cite[Theorem 3.3]{FLC} that, for a topological group $G$, the free Abelian topological group $A(G)$ is $snf$-countable if and only if $G$ contains no non-trivial convergent sequences. Therefore, by \cite[Theorems 3.1 and 3.3]{FLC}, the following corollary gives a partial answer to Problem~\ref{pr0}.

\begin{corollary}
Let $X$ be a non-separable $k_{\omega}$-space without
non-trivial convergent sequences. If $X$ is homeomorphic to a topological group, then $A(X)$ does not have a suitable set.
\end{corollary}

However, the following question is unknown for us.

\begin{question}\label{q0}
Let $G$ be a non-separable $k_{\omega}$-topological group without
non-trivial convergent sequences. Does $A(G)^{n}$ have a suitable set for some $n\geq 2$?
\end{question}

Recall that a subset $Y$ of a space $X$ is called {\it $C^{\ast}$-embedded} in $X$ if any bounded continuous
real-valued function on $Y$ extends to a continuous real-valued function on $X$. If $X$ is a compact space in which every
countable discrete subset is $C^{\ast}$-embedded, then it follows from \cite[Theorem 2.1.4]{TKA1997} that $X$ does not contain any non-trivial convergent sequences. Therefore, the following result is a partial answer to Question~\ref{q0}.

\begin{theorem}
Let $X$ be a non-separable compact space in which every
countable discrete subset is $C^{\ast}$-embedded. Then $A(X)^{n}$ does not have a suitable set for any $n\in\mathbb{N}$.
\end{theorem}

In Theorem~\ref{tttt}, the assumption of ``non-separable'' is necessary since each free topological group over a separable space has a closed suitable set, see \cite{CM}. However, if $X$ is a Lindel\"{o}f space and $F(X)$ or $A(X)$ has a closed suitable set, does $X$ is separable? The following Theorem~\ref{ccc} gives a partial answer to this question. First, we give a lemma.

\begin{lemma}\label{t0}
Let $G$ be a Lindel\"{o}f group. Then the following statements hold.

\smallskip
(a) If $G\in\mathscr{S}_{c}$, then $G$ is separable.

\smallskip
(b) If $G$ is separable and non-compact then $G\in\mathscr{S}_{c}$.

\smallskip
(c) The group $G\in\mathscr{S}_{cg}$ if and only if $G$ is countable.
\end{lemma}

\begin{proof}
(a) Assume $S$ is a closed suitable set for $G$. Then $S$ is Lindel\"{o}f since $S$ is closed in a Lindel\"{o}f group $G$. Hence $S$ is countable, which shows that $\langle S \rangle$ is countable. Thus $G$ is separable.

(b) If $G$ is separable and non-compact, then $G$ is not pseudocompact since a regular Lindel\"{o}f pseudocompact space is compact. Now it follows from \cite[Corollary 3.9]{DTT} that $G\in\mathscr{S}_{c}$.

(c) By the proof of (a), if $G\in\mathscr{S}_{cg}$ then $G$ is countable. The converse follows from \cite[Theorem 2.2]{CM}.
\end{proof}

\begin{theorem}\label{ccc}
Let $X$ be a regular space such that $X^{n}$ is Lindel\"{o}f for each $n\in\mathbb{N}$. Then the following conditions are equivalent:
\begin{enumerate}
\item $G(X)\in\mathscr{S}_{c}$.

\smallskip
\item $X$ is separable.
\end{enumerate}
\end{theorem}

\begin{proof}
Obviously, we may assume that $X$ is nonempty. Since $X$ is a regular space such that $X^{n}$ is Lindel\"{o}f for each $n\in\mathbb{N}$, it follows from \cite[Corollary 7.1.18]{A2008} that $G(X)$ is Lindel\"{o}f.

(1) $\Rightarrow$ (2). Assume $G(X)\in\mathscr{S}_{c}$. Hence $G(X)$ is separable by (a) of Lemma~\ref{t0}, then $X$ is separable by Lemma~\ref{p0}.

(2) $\Rightarrow$ (1). Assume $X$ is separable, it is easily checked that $G(X)$ is separable. Since $X\neq\emptyset$, it follows from \cite[Proposition 7.1.12]{A2008} that $G(X)$ not compact. Now by (b) of Lemma~\ref{t0}, we conclude that $G(X)\in\mathscr{S}_{c}$.
\end{proof}

Recall that a space $X$ is {\it cosmic} if it has a countable network.

\begin{corollary}
Let $X$ be a cosmic space. Then $G(X)\in\mathscr{S}_{c}$. In particular, if $X$ is a separable metric space, then $G(X)\in\mathscr{S}_{c}$.
\end{corollary}

The free topological group $G(X)$ is not a discrete for any non-discrete space $X$, hence $G(X)$ does not have any isolated points. Finally, we prove that each maximal topological group has a closed suitable set, where each maximal topological group is a space without isolated points by the following definition.

Let $X$ be a space without isolated points. Then $X$ is said to be {\it maximal} if it has isolated points in each stronger topology. A topological group is said to be {\it maximal} if its underlying space is maximal.

\begin{theorem}\label{t3}
Each maximal topological group has a closed suitable set.
\end{theorem}

\begin{proof}
Let $G$ be a maximal topological group. Then $G$ has an open countable Boolean subgroup $H$ by \cite{Mal}. Then $H$ has a closed suitable set $S$ by \cite[Theorem 2.2]{CM}. Let $A$ select one point
from each coset of $H$ in $G\setminus H$. Then $A\cup S$ is a closed suitable set for $G$.
\end{proof}

\bigskip
\section{linearly orderable topological groups with a suitable set}
In this section, we mainly provide some partial answers to Problem~\ref{pr1}, and give an affirmative answer to Problem~\ref{pr6}.

Let $G$ be a linearly orderable topological group. If $G$ has a nontrivial component, then $G$ is metrizable by \cite[Theorem 2.4]{VRS}, hence $G$ has a suitable set. Moreover, each totally-disconnected linearly orderable topological group, which is separable or has a countable pseudocharacter, is metrizable by \cite[Theorems 2.6 and 2.7]{VRS}. Therefore, Problem~\ref{pr1} is equivalent to the following Problem~\ref{pr2}.

\begin{problem}\label{pr2}
Let $G$ be a totally-disconnected, linearly orderable topological group. If $d(G)>\omega$ and $\psi(G)>\omega$, does $G$
have a (closed) suitable set?
\end{problem}

In \cite[Theorem 2.4]{ST2002}, the authors proved that each dense subgroup $H$ of a linearly orderable topological group $G$ has a (closed) suitable set if $G$ has a (closed) suitable set. Indeed, there exist a compact topological group $G$ and a dense subgroup $H$ of $G$ such that $H$ does not have any suitable set, see \cite[Theorem 2.12]{DTT}. Therefore, the following question is interesting.

\begin{question}\label{q111}
Find the property $\mathcal{P}$ such that, for any topological group $G$ with $\mathcal{P}$, then $G$ has suitable set if and only if each dense subgroup $H$ of $G$ has a suitable set.
\end{question}

Next we give a simple proof of \cite[Theorem 2.4]{ST2002}. Indeed, we prove a more general fact, which gives a partial answer to Question~\ref{q111}. First, we give some concepts.

Recall that a space is called {\it collectionwise Hausdorff} if each closed
discrete subset of $X$ can be separated by a discrete family of open neighborhoods. Clearly, each paracompact space is collectionwise Hausdorff. A subset $D$ of a space $X$ is said to be {\it transfinite sequentially dense} if each point $x\in X$ is the unique accumulation point of a transfinite sequence of points in $D$.

\begin{theorem}\label{tt56}
Let $G$ be a hereditarily collectionwise Hausdorff topological group with
a (closed) suitable set $S$ and $H$ be a transfinite sequentially dense subgroup of $G$. Then $H$ has a (closed) suitable set.
\end{theorem}

\begin{proof}
Since $G$ is hereditarily collectionwise Hausdorff, it follows that $G\setminus\{e\}$ is collectionwise Hausdorff, hence, for each $x\in S$, there exists an open neighborhood $U_{x}\subset G\setminus\{e\}$ of $x$ such that the family $\{U_{x}: x\in S\}$ is discrete in $G\setminus\{e\}$. Now put $S_{x}=\{x\}$ whenever $x\in H\cap S$. If $x\in S\setminus H$, let $S_{x}$ be a transfinite sequence of $U_{x}\cap H$ such that $x$ is the unique accumulation point of $S_{x}$ in $G$. Put $S^{\ast}=\bigcup_{x\in S}S_{x}$. Then $S^{\ast}$ is closed and discrete in $G\setminus\{e\}$ since the family $\{S_{x}: x\in S\}$ is a locally finite family and each $S_{x}$ is discrete and closed in $G\setminus\{e\}$. It is obvious that $S\subset \overline{\langle S^{\ast} \rangle}$ and $\langle S\rangle$ is dense in $G$, it follows that $\langle S^{\ast} \rangle$ is dense in $H$. Therefore, $S^{\ast}$ is a suitable set for $H$.

Clearly, if $S$ is closed in $G$, then $S^{\ast}$ is closed in $H$.
\end{proof}

\begin{proposition}\label{pp3}
Each dense subgroup of a linearly orderable topological group $G$ is a transfinite sequentially dense subgroup.
\end{proposition}

\begin{proof}
Let $H$ be a dense subgroup of $G$. If $G$ is metrizable, then it is obvious that $H$ is a transfinite sequentially dense subgroup. Now assume that $G$ is not metrizable. Since $G$ is a linearly orderable topological group, it follows from \cite[Theorem 6]{NR} that the neutral element of $G$ has a totally ordered local base consisting of clopen subgroups $\{U_{\alpha}: \alpha<\tau\}$, where we may assume that $\tau$ is a regular cardinal. For each $x\in G$, define a transfinite sequence in $H$ as follows. If $x\in H$, then put $S_{x}=\{x\}$; otherwise pick a transfinite sequence $S_{x}=\{z(x, \alpha): \alpha<\tau\}$, where each $z(x, \alpha)\in H\cap (xU_{\alpha+1}\setminus xU_{\alpha})$ with $\alpha<\tau$. Clearly, $x$ is the unique accumulation point of $S_{x}$ in $H$. Therefore, $H$ is a transfinite sequentially dense subgroup.
\end{proof}

\begin{corollary}\cite[Theorem 2.4]{ST2002}
Each dense subgroup of a linearly orderable topological group $G$ with a (closed) suitable set has a (closed) suitable set.
\end{corollary}

\begin{proof}
Since $G$ is hereditarily paracompact by \cite[Theorem 8]{NP1}, it follows $G$ is hereditarily collectionwise Hausdorff. Now the conclusion is immediate by applying Theorem~\ref{tt56} and Proposition~\ref{pp3}.
\end{proof}

Moreover, we can prove the following theorem.

\begin{theorem}
Let $G$ be a linearly orderable topological group. Then each subgroup of $G$ has a suitable set if and only if each closed linearly orderable subgroup of $G$ has a suitable set.
\end{theorem}

\begin{proof}
The sufficiency is obvious, hence is suffices to prove the necessity. Let $H$ be a subgroup of $G$. If $\overline{H}$ is metrizable, then $\overline{H}$ and $H$ all have suitable sets by \cite[Theorem 6.6]{CM}. If $\overline{H}$ is not metrizable, then $G$ is not metrizable, hence it follows from \cite{NP1} that $G$ contains a decreasing chain of open subgroups $\{U_{\alpha}: \alpha\in I\}$
which forms a local base at the identity. Then $\{U_{\alpha}\cap \overline{H}: \alpha\in I\}$ is a decreasing chain of open subgroups in $\overline{H}$
which form a local base at the identity. Hence $\overline{H}$ is a linearly orderable topological subgroup. By our assumption, $\overline{H}$ has a suitable set, then $H$ has a suitable set by \cite[Theorem 2.4]{ST2002}.
\end{proof}

{\bf Note:} By \cite[Corollary 3.3]{DTT}, spaces with a $\sigma$-discrete network (that is, $\sigma$-spaces) belong to $\mathcal{P}$ in Question~\ref{q111}. Therefore, we have the following proposition.

\begin{proposition}
Let $X$ be a paracompact $\sigma$-space. Then each subgroup of $F(X)$ has a suitable set.
\end{proposition}

\begin{proof}
By \cite[Theorem 7.6.7]{A2008}, the free topological group $F(X)$ is a paracompact $\sigma$-space, hence each subgroup of $F(X)$ is a $\sigma$-space. The conclusion is derived from \cite[Corollary 3.3]{DTT}.
\end{proof}

\begin{corollary}
Let $X$ be a metrizable space. Then each subgroup of $F(X)$ has a suitable set.
\end{corollary}

Let the direct product $\prod_{\alpha<\tau}G_{\alpha}$ endowed with the Tychonoff topology, and let $G$ the $\sigma$-product of $\prod_{\alpha<\tau}G_{\alpha}$ which inherits the subspace group topology $\mathscr{F}_{G}$, where each $G_{\alpha}$ is a non-trivial discrete group. For each $g=(g_{\alpha})_{\alpha<\tau}\in \prod_{\alpha<\tau}G_{\alpha}$, we denoted the set $\{g_{\alpha}: g_{\alpha}\neq e_{\alpha}, \alpha<\tau\}$ by $Proj(g)$.

\smallskip
Now the following theorem gives an affirmative answer to Problem~\ref{pr6}.

\begin{theorem}\label{t4}
Every subgroup of $(G, \mathscr{F}_{G})$ has a generating suitable
set.
\end{theorem}

\begin{proof}
Without loss of generality, we may assume that $G_{\alpha}\cap G_{\beta}=\emptyset$ for any distinct $\alpha, \beta\in \tau$. Let $H$ be a subgroup of $G$. Next we can construct a sequence $\{S_{n}: n\in\mathbb{N}\}$ of subsets of $H$ satisfy the following conditions:

\smallskip
(1) $S_{1}=\{h\in H: |Proj(h)|=1\}$;

\smallskip
(2) for any $n\in\mathbb{N}\setminus\{1\}$, put $$S_{n}=\{h\in H: |Proj(h)|=n, Proj(g)\setminus Proj(h)\neq\emptyset\ \mbox{for any}\ g\in\bigcup_{i=1}^{n-1}S_{i}\}.$$

Now put $S=\bigcup_{n\in\mathbb{N}}S_{n}$. We conclude that $S$ is a generating suitable set of $H$. We divide the proof into the following three steps.

\smallskip
{\bf Step 1}: $S$ is a discrete subset.

\smallskip
Indeed, for any $h\in S$, there exists $n\in\mathbb{N}$ such that $h\in S_{n}$. Let $$Proj(h)=\{h_{\alpha_{1}}, \ldots, h_{\alpha_{n}}\},$$ and put $$U=(\prod_{i=1}^{n}\{h_{\alpha_{i}}\})\times \prod_{\alpha\in\tau\setminus\{\alpha_{i}: 1\leq i\leq n\}}G_{\alpha}.$$ We claim that $U\cap S=\{h\}$. Suppose not, from the construction of $S_{n}$, it follows there exist $m>n$ and $g\in U\cap S$ such that $g\in S_{m}$. Since $g\in U$, we have $Proj(h)\subset Proj(g)$, which is a contradiction because $Proj(s)\setminus Proj(g)\neq\emptyset$ for any $s\in S_{n}$ by the definition of $S_{m}$. Therefore, $U\cap S=\{h\}$.

\smallskip
{\bf Step 2}: $S\cup\{e\}$ is closed in $H$.

\smallskip
Take any $h\in H\setminus(S\cup\{e\})$, and let $Proj(h)=\{h_{\beta_{1}}, \ldots, h_{\beta_{l}}\}$, where $l\in\mathbb{N}$. Obviously, we have $h\not\in S_{l}$, then there exist $j<l$ and $s\in S_{j}$ such that $Proj(s)\subset Proj(h)$. Put $$V=(\prod_{i=1}^{l}\{h_{\beta_{i}}\})\times \prod_{\alpha\in\tau\setminus\{\beta_{i}: 1\leq i\leq l\}}G_{\alpha}.$$ Obviously, we have $V\cap (\bigcup_{i=1}^{l}S_{i})=\emptyset$. Now it suffices to prove that $V\cap (\bigcup_{i>l}S_{i})=\emptyset$. Suppose not, choose any $g\in V\cap (\bigcup_{i>l}S_{i})$. Since $g\in V$, it follows that $Proj(h)\subset Proj(g)$, then $Proj(s)\subset Proj(g)$, which is a contradiction since $s\in S_{j}$ and $g\in \bigcup_{i>l}S_{i}$. Therefore, $V\cap S=\emptyset$.

\smallskip
{\bf Step 3}: $\langle S\rangle$ is a generated set of $H$.

\smallskip
Indeed, take any $h\in H\setminus(S\cup\{e\})$, and let $Proj(h)=\{h_{\beta_{1}}, \ldots, h_{\beta_{n}}\}$, where $n\in\mathbb{N}$. By the proof of Step 2, there exist $j_{1}<n$ and $s_{1}\in S_{j_{1}}$ such that $Proj(s_{1})\subset Proj(h)$. Put $h_{1}=s_{1}^{-1}h$. Then $|Proj(h_{1})|<|Proj(h)|$. If  $h_{1}\in S$, then $h=s_{1}h_{1}$, the result is desired; otherwise, $h_{1}\not\in S$. By a similar method, there exist $j_{2}<|Proj(h_{1})|$ and $s_{2}\in S_{j_{2}}$ such that $Proj(s_{2})\subset Proj(h_{1})$. By induction, since $Proj(h)$ is a finite set, it is easy to see that there exist $s_{1}, s_{2}, \ldots, s_{k}$ for some $k\in \mathbb{N}$ such that $h=s_{1} s_{2} \ldots s_{k}$, where each $s_{i}\in \bigcup_{j=1}^{n-1}S_{j}$ and $\sum_{i=1}^{k}|Proj(s_{i})|=|Proj(h)|$. Therefore, we have $\langle S\rangle=H.$

Therefore, $S$ is a generating suitable set of $H$.
\end{proof}

Suppose that $\tau$ is an infinite cardinal and $G^{\star}$ is the group $\prod_{\alpha<\tau}G_{\alpha}$ endowed with the linear topology such that the neutral element has a base which consists of the subgroups $U_{\alpha}=\prod_{\alpha\leq\beta<\tau}G_{\beta}$, where $\alpha<\tau$. For any $\alpha<\tau$, let $\pi_{\alpha}: G^{\star}\rightarrow G[\alpha]$ be the natural homomorphic retraction of $G$ onto its subgroup $G[\alpha]=\prod_{\beta<\alpha}G_{\beta}$. In \cite{ST2002}, the authors posed the following problem.

\begin{problem}\cite[Problem 4.4]{ST2002}\label{pr7}
It is true that every subgroup of $G^{\star}$ has a suitable set?
\end{problem}

The following theorem gives a partial answer to Problems~\ref{pr1} and~\ref{pr7}.

\begin{theorem}
Assume that the group $G=\prod_{\alpha<\tau}G_{\alpha}$ carries the linear group topology $G^{\star}$, and $H$ is a subgroup of $G$. If $|\pi_{\alpha}(H)|=|H|$ for some $\alpha<\tau$, then $H$ has a closed generating suitable set.
\end{theorem}

\begin{proof}
For each basic open subgroup $U_{\alpha}=\prod_{\alpha\leq\beta\leq\tau}G_{\beta}$, the group $H$ cannot be covered by less than $|\pi_{\alpha}(H)|=|H/(H\cap U_{\alpha})|$ translates of its open subgroup $H\cap U_{\alpha}$; hence we have $|H|=|\pi_{\alpha}(H)|<b(H)$. Now, it follows from \cite[Lemma 3.5]{ST2002} that $H$ has a closed generating
suitable set.
\end{proof}

Indeed, by a similar remark of \cite[Remark 3.1]{ST2002}, we may assume that $\tau$ is regular. Let $\sum_{\alpha<\tau}G_{\alpha}=\{x\in G^{\star}: |\mbox{supp}(x)|\leq\aleph_{0}\}$, where $\mbox{supp}(x)=\{\alpha<\tau: x(\alpha)\neq e_{G_{\alpha}}\}$. In this section, we always assume  $\sum_{\alpha<\tau}G_{\alpha}$ with the subspace topology of $G^{\star}$, which is denoted by $G^{\star}(\omega)$. The following question is natural.

\begin{question}\label{q222}
It is true that every subgroup of $G^{\star}(\omega)$ has a suitable set?
\end{question}

Finally, we give answers to Question~\ref{q222} for the cases $\tau=\omega$ and $\tau>\omega_{1}$ respectively. First, we have the following propositions.

\begin{proposition}\label{pr9}
The group $G^{\star}(\omega)$ has a generating suitable set $S=\bigcup_{\alpha<\tau}G[\alpha]$, where $G[\alpha]=\sum_{\beta<\alpha}G_{\beta}$ for each $\alpha<\tau$.
\end{proposition}

\begin{proof}
Clearly, $S$ is algebraically generates the group $G^{\star}(\omega)$. Next, we prove that $S$ is discrete and $S\cup\{e\}$ is closed in $G^{\star}(\omega)$. Indeed, fix any $\alpha<\tau$. Then $G[\alpha]$ is discrete in $G^{\star}(\omega)$ since $G[\alpha]\cap U_{\alpha}=\{e\}$, so $G[\alpha]$ is closed in $G^{\star}(\omega)$. Then, $S\setminus U_{\alpha}\subset G[\alpha]$, hence $S\setminus U_{\alpha}$ is closed and discrete in $G^{\star}(\omega)$. Since $G^{\star}(\omega)$ is Hausdorff, it follows that $S$ has no accumulations points in $G^{\star}(\omega)$ except for $e$. Therefore, $S$ is a generating suitable set for $G^{\star}(\omega)$.
\end{proof}

By a similar proof of \cite[Proposition 3.3]{ST2002}, we have the following proposition.

\begin{proposition}\label{pr8}
The group $G^{\star}(\omega)$ is topologically orderable; thus it is hereditary paracompact.
\end{proposition}

By \cite[Theorem 2,4]{ST2002}, Propositions~\ref{pr9} and~\ref{pr8}, we have the following theorem.

\begin{theorem}
Each dense subgroup of $G^{\star}(\omega)$ has a suitable set.
\end{theorem}

The following lemma is similar to \cite[Lemma 3.7]{ST2002}, so we omit to the proof.

\begin{lemma}\label{l11}
Let $H$ be a subgroup of $G^{\star}(\omega)$. Then, for each $\alpha<\tau$ and $x\in\pi_{\alpha}(H)$, there exists a closed discrete subset $F_{\alpha, x}$ of $H\cap \pi_{\alpha}^{-1}(x)$ such that $|F_{\alpha, x}|\leq\tau$ and for any $\beta<\tau$, $H\cap xU_{\beta}\neq\emptyset$ if and only if $F_{\alpha, x}\cap xU_{\beta}\neq\emptyset$. Moreover,

\smallskip
(i) if $x\in \overline{H}\setminus H$, then $x$ is the unique accumulation point of $F_{\alpha, x}$ in $G^{\star}(\omega)$.

\smallskip
(ii)  if $H$ is closed, then we can choose $F_{\alpha, x}$ such that $|F_{\alpha, x}|<\tau$.
\end{lemma}

\begin{lemma}\label{l33}
Let $H$ be a closed subgroup of $G^{\star}(\omega)$, where $\tau>\omega_{1}$. Suppose that $|\pi_{\alpha}(H)|<\tau$ for any $\alpha<\tau$ and $D=\{d_{\alpha}: \alpha<\tau\}$ is a subset of $H$ such that $\langle D\rangle=H$. There is a closed unbounded subset $B$ of $\tau$ such that the following conditions are satisfied for any $\alpha\in B$:

\smallskip
(a) $\mbox{supp}(d_{\gamma})\subset \alpha$ whenever $\gamma<\alpha$;

\smallskip
(b) $\pi_{\alpha}(H)=\langle D_{\alpha}\rangle$, where $D_{\alpha}=\{d_{\gamma}: \gamma<\tau\}$.
\end{lemma}

\begin{proof}
Let $B$ be the set of all ordinals $\alpha<\tau$ which satisfy (a) and (b) of the lemma. Clearly, $B$ is closed. Next, we prove that $B$ is unbounded. For each $\alpha<\tau$, put $D_{\alpha}=\bigcup\{F_{\alpha, x}: x\in \pi_{\alpha}(H)\}$, where each $F_{\alpha, x}$ satisfies (ii) of Lemma~\ref{l11}. First, following claim is obvious since $|\pi_{\alpha}(H)|<\tau$.

{\bf Claim:} For each $\alpha<\tau$, the cardinality of $D_{\alpha}$ is less than $\tau$ and there exists an ordinal $\alpha^{\ast}$ with $\alpha<\alpha^{\ast}<\tau$ such that $D_{\alpha}\subset \langle D_{\alpha^{\ast}}\rangle$ and $\mbox{supp}(d_{\beta})\subset\alpha^{\ast}$ for each $\beta\leq\alpha$.

Now take any ordinal $\alpha_{0}<\tau$. By Claim, we can define a sequence $\{\alpha_{\gamma}: \gamma\in\omega_{1}\}$ such that the following conditions hold:

\smallskip
(1) $\alpha_{\gamma}<\tau$ for each $\gamma\in\omega_{1}$;

\smallskip
(2) if $\gamma$ has a predecessor, then $\alpha_{\gamma}=\alpha_{\gamma-1}^{\ast}$;

\smallskip
(3) if $\gamma$ is limit ordinal, then $\alpha_{\gamma}=(\bigcup_{\kappa<\gamma}\alpha_{\kappa}^{\ast})^{\ast}$.

Clearly, we have $\alpha_{\beta}>\alpha_{\gamma}$ if $\beta>\gamma$. Now put $\alpha^{\flat}=\sup\{\alpha_{\gamma}: \gamma\in\omega_{1}\}$. Then $\alpha_{0}<\alpha^{\flat}<\tau$ and $\alpha^{\flat}$ satisfies (a) of the lemma. Next we only need to prove that $\alpha^{\flat}$ satisfies (b). Indeed, take any $x\in\pi_{\alpha^{\flat}}(H)$. Then there exists $h\in H$ such that $x=\pi_{\alpha^{\flat}}(h)$. Clearly, we have $\mbox{supp}(x)\subset \alpha^{\flat}$. Since $\alpha^{\flat}=\sup\{\alpha_{\gamma}: \gamma\in\omega_{1}\}$ and $\mbox{supp}(x)$ is countable, there exists $\gamma_{0}<\omega_{1}$ such that $\mbox{supp}(x)\subset\alpha_{\gamma_{0}}$. Therefore, $x=\pi_{\alpha_{\gamma_{0}}}(h)\in \pi_{\alpha_{\gamma_{0}}}(H)$. Then, since $y\in H\cap xU_{\alpha^{\flat}}\neq\emptyset$, it follows from our choice of the set $F_{\alpha_{\gamma_{0}}, x}$ that $F_{\alpha_{\gamma_{0}}, x}\cap xU_{\alpha^{\flat}}\neq\emptyset$. Choose any point $z\in F_{\alpha_{\gamma_{0}}, x}\cap xU_{\alpha^{\flat}}$. Hence $z\in D_{\alpha_{\gamma_{0}}}\subset\langle D_{\alpha_{(\gamma_{0}+1)}}\rangle\subset\langle D_{\alpha^{\flat}}\rangle$, then $\mbox{supp}(z)\subset \alpha^{\flat}$ by (a). Moreover, since $z\in xU_{\alpha^{\flat}}$, we have $x=\pi_{\alpha^{\flat}}(z)$, thus $x=z$. Hence we conclude that $\pi_{\alpha^{\flat}}(H)=\langle D_{\alpha^{\flat}}\rangle$. Then $\alpha^{\flat}\in B$ and $\alpha_{0}<\alpha^{\flat}$, hence $B$ is unbounded in $\tau$.
\end{proof}

\begin{lemma}\label{l22}
Let $H$ be a subgroup of $G^{\star}(\omega)$ such that $|\pi_{\alpha}(H)|<|H|$ for each $\alpha<\tau$. Then, for each $\alpha<\tau$, we have $b(H\cap U_{\alpha})=|H|=b(H)$.
\end{lemma}

\begin{proof}
We first prove that $b(H)=|H|$. Indeed, for each $\alpha<\tau$, we conclude that $|\pi_{\alpha}(H)|<b(H)$ since $\pi_{\alpha}(H)$ is topologically isomorphic to $H/(H\cap U_{\alpha}).$  Hence it follows that $|H|=\sup_{\alpha<\tau}|\pi_{\alpha}(H)|\leq b(H)$. Now take any open neighborhood $U$ of $e$ in $H$. Then there exists $\beta<\tau$ such that $H\cap U_{\beta}\subset U$. Choose $D\subset H$ such that $\{d(H\cap U_{\beta}): d\in D\}$ forms a (faithfully indexed) partition of $H$. Then $H=D\cdot (H\cap U_{\beta})\subset D\cdot U\subset H$, whence it follows that $b(H)\leq H$. Therefore, $b(H)=|H|$.

By the above proof, we have $b(H\cap U_{\alpha})=|H\cap U_{\alpha}|$ for each $\alpha<\tau$. Since $|H|=|H\cap U_{\alpha}|\cdot |\pi_{\alpha}(H)|$ and $|\pi_{\alpha}(H)|<|H|$, it concludes that $|H\cap U_{\alpha}|=|H|$. Therefore, $b(H\cap U_{\alpha})=|H|=b(H)$.
\end{proof}

By Lemmas~\ref{l11},~\ref{l33},~\ref{l22}, we can prove the following lemma by the similar methods in \cite[Lemmas~3.10 and~~3.12]{ST2002}, hence we omit the proof.

\begin{lemma}\label{l44}
Let $H$ be a closed subgroup of $G^{\star}(\omega)$, where $\tau>\omega_{1}$. Then there exist a closed unbounded set $B\subset \tau$ and a representation $S=\bigcup_{\alpha<\tau}S_{\alpha}$ such that $S$ is a generating suitable set for $H$ which satisfy the following conditions for any $\alpha, \beta\in B$:
\begin{enumerate}
\item $S_{\alpha}$ is closed and discrete in $H$;

\item $S_{\alpha}\subset S_{\beta}$ if $\alpha<\beta$;

\item $S_{\alpha}=\bigcup_{\beta\in B\cap \alpha}S_{\beta}$ if $\beta$ is limit in $B$;

\item $\langle S_{\alpha}\rangle=\pi_{\alpha}(\langle S_{\alpha}\rangle)=\pi_{\alpha}(H)$;

\item if $\alpha<\beta$ and $x\in S_{\beta}\setminus S_{\alpha}$, then $\mbox{supp}(x)\cap \alpha=\emptyset$.
\end{enumerate}
\end{lemma}

By Lemma~\ref{l44}, we can prove the following lemma by the similar method in \cite[Theorem 3.16]{ST2002}, hence we omit the proof.

\begin{lemma}\label{l55}
Each non-closed subgroup of $G^{\star}(\omega)$ has a closed generated suitable set, where $\tau>\omega_{1}$.
\end{lemma}

Now we can prove the second main theorem in this section.

\begin{theorem}
Every subgroup of $G^{\star}(\omega)$ has a suitable set, where $\tau=\omega$ or $\tau>\omega_{1}$.
\end{theorem}

\begin{proof}
Let $H$ be a subgroup of $G^{\star}(\omega)$. If $\tau=\omega$, then $G^{\star}(\omega)$ is metrizable, hence $H$ is metriable, then $H$ has a suitable set. Assume$\tau>\omega_{1}$.  By \cite[Remark 3.1]{ST2002}, we may assume that $\tau$ is a regular uncountable cardinal. By Proposition~\ref{pr9}, it suffices to consider the case when $\pi_{\alpha}(H)<|H|$ for each $\alpha<\tau$, which implies that $|H|\geq\tau$. If $H$ is closed in $G$, then the conclusion follows from Lemma~\ref{l44}. Now assume that $H$ is non-closed, then it follows from Lemma~\ref{l55} that $H$ has a closed generated suitable set.
\end{proof}

The following question is still unknown for us.

\begin{question}
If $\tau=\omega_{1}$, does every subgroup of $G^{\star}(\omega)$ have a suitable set?
\end{question}

\bigskip
\section{The existence of suitable sets in topological groups with an $\omega^\omega$ }
In this section, we consider the following Question~\ref{q444} and give some partial answers to this question. We prove that, for any topological group $G$ with an $\omega^{\omega}$-base, if $G$ is a $k$-space then it has a suitable set. Moreover, we show that every linearly orderable topological group with an $\omega^{\omega}$-base is metrizable; thus it
has a suitable set.

\begin{question}\label{q444}
Let $G$ be a topological group with an $\omega^{\omega}$-base. Does $G$ have a suitable set?
\end{question}

Let $\omega^\omega$ be the set all the functions from $\omega$ to $\omega$ with the natural partial order defined
by $\alpha\leq\beta$ iff $\alpha(n)\leq \beta(n)$ for all $\alpha, \beta\in \omega^\omega$. A topological space $(X, \tau)$ is said to be have an {\it $\omega^\omega$-base} if for each point $x\in X$ there exists a neighborhood base $\{U_{\alpha}[x]: \alpha\in\omega^\omega\}$ such that $U_{\beta}[x]\subset U_{\alpha}[x]$ for all $\alpha\leq\beta$ in $\omega^\omega$. Clearly, each first-countable space with an $\omega^\omega$-base.

Clearly, there exists a topological group does not has an $\omega^{\omega}$-base. Indeed, the $\sum$-product $G$ of uncountably many copies of $\mathbb{Z}(2)$ is a Fr\'{e}chet-Urysohn topological group that is not metrizable, hence $G$ does not has an $\omega^{\omega}$-base by \cite[Theorem 2.12]{GKL2015}. The following theorem gives a generalization of \cite[Theorem 6.6]{CM}.

\begin{theorem}
Let $G$ be a topological group with an $\omega^{\omega}$-base. If $G$ is a $k$-space, then $G$ has a suitable set.
\end{theorem}

\begin{proof}
Since $G$ is a $k$-space with an $\omega^{\omega}$-base, it follows from \cite[Corollary 3.13]{GKL2015} that $G$ is metrizable or contains a submetrizable open $k_{\omega}$-subgroup. If $G$ is metrizable, then $G$ has a suitable set by \cite[Theorem 6.6]{CM}. Now assume that $G$ contains a submetrizable open $k_{\omega}$-subgroup $H$. Then $H$ has a countable network by \cite{G1984}, hence $H$ has a suitable set $S$ by \cite[Corollary 3.10]{DTT}. Choose $D\subset G\setminus H$ such that $\{dH: d\in D\}$ forms a (faithfully indexed) partition of $G\setminus H$. Then it is easily checked that $S\cup D$ is a suitable set for $G$.
\end{proof}

From well-known Birkhoff and Kakutani's Theorem, it follows that each first-countable topological group is metrizable, hence we have the following corollary.

\begin{corollary}\cite[Theorem 6.6]{CM}
Each metrizable topological group has a suitable set.
\end{corollary}

The following example shows that our generalization is significant.

\begin{example}
There exists a topological group $G$ which is a $k$-space with an $\omega^{\omega}$-base, but $G$ is not metrizable, hence it is not first-countable.
\end{example}

\begin{proof}
Let $X$ be a non-discrete submetrizable $k_{\omega}$-space. Then $F(X)$ and $A(X)$ are $k$-spaces with $\omega^{\omega}$-bases by \cite[Theorem 4.6]{LRZ2020} and \cite[Theorem 2.2]{LL2025} respectively. However, $F(X)$ and $A(X)$ are not metrizable since $X$ is non-discrete by \cite[Theorem 7.1.20]{A2008}.
\end{proof}

By \cite{TKA1997}, neither separability
nor $\sigma$-compactness imply
the existence of a suitable set in topological groups. However, for any topological group $G$ with an $\omega^{\omega}$-base, any one of the conditions of  separability, $\sigma$-compactness implies that $G$ has a suitable set.

Recall that a topological group $G$ is {\it precompact}
if for each neighborhood $U$ of the identity of $G$,
there exists a finite subset $F$ of $G$ such that $FU=G$.

\begin{theorem}
Let $G$ be a separable topological group with an $\omega^{\omega}$-base. Then $G$ has a suitable set. Further, if $G$ is not precompact, it has a closed suitable set.
\end{theorem}

\begin{proof}
If $G$ is precompact, then $G$ is metrizable by \cite[Corollary 3.11]{GKL2015}, hence $G$ has suitable set by \cite[Theorem 6.6]{CM}. Now assume that $G$ is not precompact, then $G$ has a closed suitable set by
\cite[Corollary 5.8]{CM}.
\end{proof}

\begin{theorem}
Let $G$ be a $\sigma$-compact topological group with an $\omega^{\omega}$-base. Then $G$ has a suitable set.
\end{theorem}

\begin{proof}
By \cite[Theorem 1.3]{GKL2015} and \cite{G1984}, $G$ has a countable network, hence $H$ has a suitable set $S$ by \cite[Corollary 3.10]{DTT}.
\end{proof}

The following theorem gives a partial answer to Problem~\ref{pr1}.

\begin{theorem}
Every linearly orderable topological group with an $\omega^{\omega}$-base is metrizable; thus it
has a suitable set.
\end{theorem}

\begin{proof}
Let $G$ be a linearly orderable topological group with an $\omega^{\omega}$-base. Now assume that $G$ is not metrizable. Indeed, since $G$ has an $\omega^{\omega}$-base, it follows from \cite[Theorem 6]{NR} that $G$ has a totally ordered local base $\{H_{\alpha}: \alpha<\omega_{1}\}$ consisting of clopen subgroups, where we may assume that $H_{\alpha}\neq H_{\beta}$ for any distinct $\alpha, \beta\in\omega_{1}$. Moreover, $G$ is an infinite space. Since $G$ has an $\omega^{\omega}$-base, it follows that $G$ has a countable cs$^{\ast}$-network $\{P_{n}: n\in\mathbb{N}\}$ at $e$ by \cite[Theorem 6.4.1]{B2017}. By \cite[Theorem 2.2.4]{MD}, each infinite orderable space contains a non-trivial convergent sequence. Take any non-trivial convergent sequence $\{a_{n}: n\in\mathbb{N}\}$ which converges to $e$, where $a_{n}\neq a_{m}$ for distinct $n, m\in\mathbb{N}$. Then there exist $n_{0}\in\mathbb{N}$ and a uncountable subset $A\subset \omega_{1}$ such that $P_{n_{0}}\subset H_{\alpha}$ for each $\alpha\in A$ and $P_{n_{0}}$ contains a non-trivial subsequence of $\{a_{n}: n\in\mathbb{N}\}$. Clearly, we have $\bigcap_{\alpha\in A}H_{\alpha}=\{e\}$, which implies $P_{n_{0}}=\{e\}$. This is a contradiction. Therefore, $G$ is metrizable. Thus it has a suitable set by \cite[Theorem 6.6]{CM}.
\end{proof}

However, the following question is still unknown for us.

\begin{question}
If $G$ is a generalized orderable topological group with an $\omega^{\omega}$-base, does $G$
have a suitable set?
\end{question}

Finally, the following two theorems show that some particular topological group with an $\omega^{\omega}$-base has a suitable set.

\begin{theorem}
Let $X$ be a Tychonoff space. If $C_{p}(X)$ has an $\omega^{\omega}$-base, then $C_{p}(X)$ has a suitable set.
\end{theorem}

\begin{proof}
By \cite[Proposition 4.3]{GKL2015} and \cite[Theorem 6.6]{CM}, the result is desired.
\end{proof}

We say that a topological group $G$ is {\it MAP} iff the family of continuous homomorphisms to
unitary groups separate its points.

\begin{theorem}
If $G$ is a MAP Abelian group with an $\omega^{\omega}$-base, then $G$ has a suitable set.
\end{theorem}

\begin{proof}
By \cite[Corollary 5.6]{GKL2015} and \cite[Theorem 6.6]{CM}, the result is desired.
\end{proof}

\bigskip
\section{The cardinal invariant of topological groups with an suitable set}
In this section, we mainly consider the following two questions about the the density and cardinality of closed suitable set. First, we give a notation.

For $G\in\mathscr{S}_{c}$ denote by $s_{c}(G)$ the minimum
cardinality of a closed suitable set for $G$.

\begin{question}
If $G\in\mathscr{S}_{c}$, under what conditions we have $d(G)= s_{c}(G)+\aleph_{0}$?
\end{question}

\begin{question}\label{q1}\cite[Question 1.6]{DTA}
Can the implication $G\in\mathscr{S}_{c}\Rightarrow d(G)\leq s_{c}(G)+\aleph_{0}$ be strengthened?
\end{question}

\begin{theorem}\label{t1}
If $G$ is a topological group with $G\in\mathscr{S}_{c}$ and $d(G)<b(G)$, then $d(G)= s_{c}(G)+\aleph_{0}$.
\end{theorem}

\begin{proof}
Assume that $d(G)=\kappa$. Since $d(G)<b(G)$, there exists a nonempty open subset $U$ of $G$ no $A\subset G$ with $|A|<\kappa$ satisfies $G=AU$. Therefore, by induction, we can find a set $\{b_{\gamma}: \gamma<\kappa\}\subset G$ such that $b_{\gamma}\not\in\bigcup_{\alpha<\gamma}b_{\alpha}U$ for each $\gamma<\kappa$. Now take an open neighborhood $V$ of the neutral element in $G$ such that $V=V^{-1}$ and $V^{4}\subset U$. From \cite[Lemma 1.4.22]{A2008}, it follows that $\{b_{\gamma}V: \gamma<\kappa\}$ is a disjoint family such that, for any $g\in G$, the open set $gV$ meets at most one of the sets $b_{\gamma}V$. Since $V$ is open and $d(G)=\kappa$, we conclude that $d(V)\leq\kappa$, hence there is a dense subset $\{c_{\gamma}: \gamma<\kappa\}$ of $V$. Put $$S=\{b_{\gamma}: \gamma<\kappa\}\cup\{b_{\gamma}c_{\gamma}: \gamma<\kappa\}$$ and let $H$ be the open subgroup of $G$ generated by $V\cup\{b_{\gamma}V: \gamma<\kappa\}$. It easy to see that $S$ is closed discrete and $\langle S\rangle$ is dense in $H$. Let $A$ select one point from each coset of $H$ in $G$. Clearly, $|A|\leq\kappa$ since $d(G)=\kappa$. Then $S\cup A$ is a closed suitable for $G$ and $|S\cup A|\leq\kappa$, which shows $s_{c}(G)\leq\kappa$. Therefore, we have $d(G)= s_{c}(G)+\aleph_{0}$.
\end{proof}

\begin{theorem}\label{t2}
If $G$ is a non-pseudocompact separable topological group with $G\in\mathscr{S}_{c}$, then $s_{c}(G)\leq\aleph_{0}$, that is, $d(G)=s_{c}(G)+\aleph_{0}$.
\end{theorem}

\begin{proof}
If $\omega<b(G)$, then $d(G)= s_{c}(G)+\aleph_{0}$ by Theorem~\ref{t1}, thus $s_{c}(G)\leq\aleph_{0}$. Assume that $b(G)\leq\omega$, then $G$ is precompact. Since $G$ is separable, non-pseudocompact and precompact, it follows from \cite[Lemma 3.5]{DTT} that $G$ has a countable closed suitable set, hence $s_{c}(G)\leq\aleph_{0}$.
\end{proof}

\begin{example}
There exists a non-pseudocompact separable topological group $G$ such that $G\in\mathscr{S}_{c}$, but $e(G)>\omega$.
\end{example}

\begin{proof}
Let $X$ be a uncountable separable, non-pseudocompact space with a closed discrete subset $A$ such that $|A|=|X|$. For example, we can take $X=\mathbb{S}^{2}$, where $\mathbb{S}$ is the Sorgenfrey line. Then $G=F(X)\times F(X)$ is a separable topological group with a closed discrete subset $A\times A$ with $|A\times A|=|G|$. Clearly, $G$ is separable and non-pseudocompact by \cite[Lemma 7.5.2]{A2008}. Moreover, by \cite[Lemma 3.1]{DTT}, $G$ has a closed suitable set $S$ such that $\langle S \rangle=G$. Then it follows from Theorem~\ref{t2} that $s_{c}(G)\leq\aleph_{0}$. However, it is obvious that $e(G)>\omega$.
\end{proof}

\begin{theorem}
For each maximal topological group $G$, we have $d(G)=s_{c}(G)+\aleph_{0}$.
\end{theorem}

\begin{proof}
By the proof of Theorem~\ref{t3}, we have $|A|\leq d(G)$, hence $d(G)=s_{c}(G)+\aleph_{0}$.
\end{proof}

\end{document}